\documentclass[12pt]{amsart}
\usepackage{amssymb,verbatim,graphicx}

\usepackage{ulem}  

\usepackage{graphics,epsfig,psfrag} 
\usepackage{pb-diagram,hyperref,graphics,epsfig,psfrag}
\usepackage[displaymath,mathlines,pagewise]{lineno}
\newcommand*\patchAmsMathEnvironmentForLineno[1]{%
  \expandafter\let\csname old#1\expandafter\endcsname\csname #1\endcsname 
  \expandafter\let\csname oldend#1\expandafter\endcsname\csname end#1\endcsname 
  \renewenvironment{#1}%
     {\linenomath\csname old#1\endcsname}%
     {\csname oldend#1\endcsname\endlinenomath}}%
\newcommand*\patchBothAmsMathEnvironmentsForLineno[1]{%
  \patchAmsMathEnvironmentForLineno{#1}%
  \patchAmsMathEnvironmentForLineno{#1*}}%
\AtBeginDocument{%
\patchBothAmsMathEnvironmentsForLineno{equation}%
\patchBothAmsMathEnvironmentsForLineno{align}%
\patchBothAmsMathEnvironmentsForLineno{flalign}%
\patchBothAmsMathEnvironmentsForLineno{alignat}%
\patchBothAmsMathEnvironmentsForLineno{gather}%
\patchBothAmsMathEnvironmentsForLineno{multline}%
}
\usepackage{color}
\usepackage{braket,tikz-cd}
\definecolor{darkred}{rgb}{0.5,0,0}
\definecolor{darkgreen}{rgb}{0,0.5,0}
\definecolor{darkblue}{rgb}{0,0,0.5}
\hypersetup{ colorlinks,
linkcolor=darkblue,
filecolor=darkgreen,
urlcolor=darkred,
citecolor=darkblue }
\newtheorem{theorem}{Theorem}[section]

\newtheorem{corollary}[theorem]{Corollary}

\newtheorem{proposition}[theorem]{Proposition}
\newtheorem{lemma}[theorem]{Lemma}
\newtheorem{lem}[theorem]{}
\theoremstyle{definition}
\newtheorem{definition}[theorem]{Definition}
\theoremstyle{remark}
\newtheorem{remark}[theorem]{Remark}
\newtheorem{example}[theorem]{Example}
\newcommand{\blem}{\begin{lem} \rm}
\newcommand{\elem}{\end{lem}}

\newcommand\M{\mathcal{M}}

\newcommand{\J}{\mathcal{J}}

\newcommand{\U}{\mathcal{U}}

\newcommand{\CC}{C\kern-1.3ex|}

\newcommand{\Z}{\mathbb{Z}}
\newcommand{\Q}{\mathbb{Q}}

\newcommand{\on}{\operatorname}

\newcommand{\univ}{\on{univ}}

\newcommand{\Ver}{\on{Vert}}

\newcommand\cU{\mathcal{U}}
\newcommand\cM{\mathcal{M}}
\newcommand\cN{\mathcal{N}}

\newcommand{\s}{\on{s}}
\newcommand{\us}{{\on{us}}}

\newcommand{\Aut}{ \on{Aut} }

\renewcommand{\ker}{ \on{ker}}

\newcommand\dirac{/\kern-1.2ex\partial} 
\newcommand\qu{/\kern-.7ex/} 
\newcommand\lqu{\backslash \kern-.7ex \backslash} 

\newcommand\dr{r_+ \kern-.7ex - \kern-.7ex r_-}
 
\newcommand{\labell}\label

\renewcommand{\d}{{\on{d}}}
\newcommand{\ol}{\overline}
\newcommand{\olp}{\ol{\partial}}

\newcommand\eps{\epsilon}

\newcommand{\ti}{\tilde}

\newcommand\Map{\on{Map}}

\newcommand\ul{\underline}

\newcommand\reg{{\on{reg}}}
\newcommand\bdefn{\begin{definition}}
\newcommand\edefn{\end{definition}}
\newcommand\bea{\begin{eqnarray*}}
\newcommand\eea{\end{eqnarray*}}
\newcommand\bcv{\left[ \begin{array}{r} }
\newcommand\ecv{\end{array} \right] }

\newcommand\bma{\left[ \begin{array}{l} }
\newcommand\ema{\end{array} \right]}
\newcommand\ben{\begin{enumerate}}
\newcommand\een{\end{enumerate}}
\newcommand\beq{\begin{equation}}
\newcommand\eeq{\end{equation}}
\newcommand\bex{\begin{example}}
\newcommand\bsj{\left\{ \begin{array}{rrr} }
\newcommand\esj{\end{array} \right\}}

\newcommand\eex{\end{example}}

\newcommand\sx{*\kern-.5ex_X}

\newcommand{\cw}[1]{{#1}}
\newcommand{\cws}[1]{{}}

\def\mathunderaccent#1{\let\theaccent#1\mathpalette\putaccentunder}
\def\putaccentunder#1#2{\oalign{$#1#2$\crcr\hidewidth \vbox
to.2ex{\hbox{$#1\theaccent{}$}\vss}\hidewidth}}

\makeatletter 
\@namedef{dgo@dd}{\let\dg@VECTOR=\dg@twoarrowedvector}%

\newcount\dg@XSHIFT 
\newcount\dg@YSHIFT 
\def\dg@twoarrowedvector(#1,#2)#3{%
   \begingroup 
   \dg@XTEMP=#1\relax\multiply\dg@XTEMP\m@ne\relax 
   \dg@YTEMP=#2\relax\multiply\dg@YTEMP\m@ne\relax 
   \dg@ZTEMP=#1\relax 
   \ifnum\dg@ZTEMP<\z@
     \multiply\dg@ZTEMP\m@ne\relax \fi 
   \ifnum\dg@YTEMP<\z@
     \advance\dg@ZTEMP by -\dg@YTEMP 
   \else \advance\dg@ZTEMP by \dg@YTEMP \fi 
   \dg@XSHIFT=#2\relax\multiply\dg@XSHIFT\m@ne\relax\multiply\dg@XSHIFT\twoarrowsep\relax 
     \divide\dg@XSHIFT by \dg@ZTEMP\relax 
   \dg@YSHIFT=#1\relax\multiply\dg@YSHIFT\twoarrowsep\relax\divide\dg@YSHIFT by \dg@ZTEMP\relax 
   \begin{picture}(0,0)%
      \thinlines 
      \put(-\dg@XSHIFT,-\dg@YSHIFT){\vector(#1,#2){#3}}%
      \put(\dg@XSHIFT,\dg@YSHIFT){\line(#1,#2){#3}}%
      \put(\dg@XSHIFT,\dg@YSHIFT){\vector(\dg@XTEMP,\dg@YTEMP){0}}
   \end{picture}%
   \endgroup}%
\makeatother 
\newcount\twoarrowsep 
\usepackage{stackengine,graphicx}
\newsavebox{\foobox}
\newcommand{\slantbox}[2][.5]{\mbox{%
        \sbox{\foobox}{#2}%
        \hskip\wd\foobox
        \pdfsave
        \pdfsetmatrix{1 0 #1 1}%
        \llap{\usebox{\foobox}}%
        \pdfrestore
}}
\def\sfvarfhash{%
  \stackinset{c}{}{b}{1.5pt}{\rule{1.45ex}{\sfrlthk}\kern1pt}{%
  \stackinset{c}{}{b}{3.7pt}{\rule{1.45ex}{\sfrlthk}\kern1pt}{%
    \textsf{\itshape ff}%
  }}%
}
\def\sfvarhash{%
  \kern1pt%
  \stackinset{c}{}{b}{1.5pt}{\rule{1.45ex}{\sfrlthk}\kern1pt}{%
  \stackinset{c}{}{b}{3.7pt}{\rule{1.45ex}{\sfrlthk}\kern1pt}{%
    \slantbox[.2]{\mysfrule\kern2pt\mysfrule\kern2.3pt}%
  }}%
  \kern1pt%
}
\def\sfrlthk{.17ex}
\def\mysfrule{\rule{\sfrlthk}{1.4ex}}
\def\varfhash{%
  \kern -1.5pt%
  \stackinset{c}{}{c}{-1.7pt}{\kern1pt\rule{1.7ex}{\rlthk}}{%
  \stackinset{c}{}{c}{1.7pt}{\kern1pt\rule{1.7ex}{\rlthk}}{%
    \kern2pt\itshape ff\kern2pt%
  }}%
}
\def\varhash{%
  \kern-.5pt%
  \stackinset{c}{}{c}{-1.7pt}{\kern1pt\rule{1.7ex}{\rlthk}}{%
  \stackinset{c}{}{c}{1.7pt}{\kern1pt\rule{1.7ex}{\rlthk}}{%
    \kern2pt\kern.2pt\slantbox[.2]{\myrule\kern2.4pt\myrule}\kern.2pt\kern2pt%
  }}%
  \kern1pt%
}
\def\rlthk{.13ex}
\def\myrule{\rule[-.33ex]{\rlthk}{1.8ex}}

\begin{document} 

\title{Partly-local domain-dependent almost complex structures}

\author{Chris Woodward}

\address{Mathematics-Hill Center,
Rutgers University, 110 Frelinghuysen Road, Piscataway, NJ 08854-8019,
U.S.A.}  \email{j.palmer@math.rutgers.edu}

\author{Guangbo Xu}

\address{Simons Center for Geometry and Physics, State University
of New York, Stony Brook, NY 11794-3636, U.S.A.} \email{gxu@scgp.stonybrook.edu}

\begin{abstract}
  We fill a gap pointed out by Nick Sheridan in the proof of
  independence of genus zero Gromov-Witten invariants from the choice
  of divisor in the Cieliebak-Mohnke perturbation scheme
 \cite{cm:trans}.
\cws{ specifically the independence of the invariants
  from the choice of stabilizing divisor.}
\end{abstract}

\maketitle

\tableofcontents

\section{Introduction}

A convenient scheme to regularize moduli spaces of genus zero
pseudoholomorphic maps was introduced by Cieliebak-Mohnke
\cite{cm:trans}. In this scheme one chooses a Donaldson hypersurface
to stabilize the domains and then a generic domain-dependent almost
complex structure to achieve regularity.  The scheme leads to a
definition of genus zero Gromov-Witten invariants over the rationals
that counts some geometric objects, as opposed to the more abstract
perturbation schemes in the Kuranishi or polyfold methods.  The proof
of independence of Gromov-Witten invariants from the choice of
Donaldson hypersurfaces in \cite[8.18]{cm:trans} depends on the
construction of a parametrized moduli space for the following
situation: Given a type $\Gamma$ of stable marked curve let
$\ol{\cU}_\Gamma \to \ol{\cM}_\Gamma$ denote the universal curve over
the compactified moduli space $\ol{\cM}_\Gamma$ of curves of type
$\Gamma$.  Let $\J_\tau(X,\omega)$ denote the space of $\omega$-tamed
almost complex structures on the given symplectic manifold
$(X,\omega)$ with rational symplectic class
$[\omega] \in H^2(X,\omega)$.  A {\it domain-dependent almost complex
  structure} is a map
\[ J_\Gamma: \ol{\cU}_\Gamma \to J_\tau(X,\omega) .\]
Associated to a coherent collection of sufficiently generic choices
$\ul{J} = (J_\Gamma)$ is a Gromov-Witten pseudocycle
$\ol{\M}_{0,n}(X,\beta) \subset X^n$ for each number of markings $n$
and each class $\beta \in H_2(X)$.

Naturally one wishes to show that the resulting pseudocycle is
independent, up to cobordism between pseudocycles, from the choice of
Donaldson hypersurface.  Suppose that $V', V'' \subset X$ are two
Donaldson hypersurfaces and
$J' = (J'_{\Gamma'}), J'' = (J''_{\Gamma''})$ are two collections of
domain dependent almost complex structures depending on the
intersection points with $V'$ resp. $V''$, depending on some
combinatorial type $\Gamma'$ resp. $\Gamma''$.  Consider the pullback
\[ (f'')^* J'_{\Gamma'}, \ (f')^* J''_{\Gamma''}: 
\ol{\U}_\Gamma \to \J(X,V',V'') \] 
to a common universal curve $\ol{\U}_\Gamma$ for some type $\Gamma$
recording both sets of markings (so that if $\Gamma'$ resp. $\Gamma''$
has $n'$ resp. $n''$ leaves then $\Gamma$ has $n'+n''$ leaves).  One
wishes to construct a homotopy between
$(f'')^* J'_{\Gamma'}, \ (f')^* J''_{\Gamma''}$ to construct a
cobordism between the corresponding pseudocycles $\ol{\M}'_n(X,\beta)$ and
$\ol{\M}_n''(X,\beta)$ .  Unfortunately, as pointed out by Nick
Sheridan, the pullbacks
$(f'')^* J'_{\Gamma'}, \ (f')^* J''_{\Gamma''}$ do not satisfy the
locality condition used to show compactness. 
\cw{That is, the restriction
of the almost complex structures $(f')^* J_{\Gamma''}''$ (or
$(f'')^* J_{\Gamma'}'$) to some irreducible component $C_v$ of the
domain curve $C$ are not independent of markings on other components
$C_{v'} \neq C_v$, because collapsed components $C_v$ may map to
non-special points $f'(C_v) = \{ w \} \in f'(C)$ under the forgetful
map $f'$.}

In this note we modify the definition of the locality on the collapsed
components so that one may homotope between the two domain-dependent
almost complex structures without losing compactness. 
Instead of directly homotoping between the given pull-backs, one first
homotopes each pullback to an almost complex structure that is equal
to a base almost complex structure near any special point.

\section{Partly local perturbations}

We introduce the following notation for stable maps with two types of
markings.  Let $\Gamma$ be a combinatorial type of genus zero stable
curve with $n = n'+n''$ markings.  Let $V',V''$ be Donaldson
hypersurfaces in the symplectic manifold $(X,\omega)$, that is,
symplectic hypersurfaces representing large multiples $k' [\omega]$
resp.  $k'' [\omega]$ of the symplectic class
$[\omega] \in H^2(X,\Q)$.  Suppose $V'$ and $V''$ intersect
transversely.  Let $\mathcal{J} (X, V', V'')$ be the space of
$\omega$-tamed almost complex structures on $X$ that make $V'$ and
$V''$ almost complex. Let
$\J^E(X,V',V'') \subset \mathcal{J} (X, V', V'')$ be some contractible
subset of almost complex structures $J: TX \to TX$ preserving $TV'$
and $TV''$ taming the symplectic form $\omega$ and so that any
non-constant pseudoholomorphic $J$-holomorphic map $u: C \to X$ with
some given energy bound $E(u) < E$ to $X$ meets $V',V''$ each in at
least three but finitely many distinct points
$u^{-1}(V'), u^{-1}(V'')$ in the domain $C$ as in
\cite[8.18]{cm:trans}.  Let
\[ J_{V',V''} \in \bigcap_E  \J^E(X,V',V'') \] 
be a base almost complex structure that satisfies these conditions
without restriction on the energy of the map $u: C \to X$.  

The universal curve breaks into irreducible components corresponding
to the vertices of the combinatorial type.  Let
$\ol{\cU}_\Gamma \to \ol{\cM}_\Gamma$ be the closure of the universal
curve of type $\Gamma$.  For each vertex $v \in \Ver(\Gamma)$ let
$\Gamma(v)$ denote the tree with the single vertex $v$ and edges those
of $\Gamma$ meeting $v$.  Let
$\ol{\cU}_{\Gamma,v} \subset \ol{\cU}_\Gamma$ be the component
corresponding to $v$, obtained by pulling back $\ol{\cU}_{\Gamma(v)}$
so that $\cU_\Gamma$ is obtained from the disjoint union of the curves
$\cU_{\Gamma,v} \to \cM_\Gamma$ by identifying at nodes.

Cieliebak-Mohnke \cite{cm:trans} requires that the almost complex
structure is equal to the base almost complex structure near the
nodes.  This condition is not true for domain-dependent almost complex
structures pulled back under forgetful maps, and so must be relaxed as
follows.  Recall that Knudsen's (genus zero) universal curve
$\ol{\cU}_\Gamma$ \cite{kn:proj2} is a smooth projective variety, and in
particular a complex manifold.  A {\it domain-dependent almost complex
  structure} for type $\Gamma$ of stable genus zero curve is an almost
complex structure
\[ J_\Gamma:  T(\ol{\cU}_\Gamma \times X ) \to 
T(\ol{\cU}_\Gamma \times X)  \]  
that preserves the splitting of the tangent bundle
$T(\ol{\cU}_\Gamma \times X)$ into factors
$T\ol{\cU}_\Gamma \times TX$ and that is equal to the standard complex
structure on the tangent space to the projective variety
$\ol{\cU}_{\Gamma}$, and gives rise to a map from $\ol{\cU}_\Gamma$ to
$\J(X,V',V'')$ with the same notation $J_\Gamma$.  Let
\[ \J_\Gamma^E(X,V',V'') \subset \Map(\ol{\cU}_\Gamma,
\J^E(X,V',V'')) \]
denote the space of such maps taking values in $\J^E(X,V',V'')$.  With
this definition, the standard proof of Gromov convergence applies: Any
sequence $u_\nu: C_\nu \to X$ of $J_\Gamma$-holomorphic maps with
energy $E(u) <E $ may be viewed as a finite energy sequence of maps to
$\ol{\cU}_\Gamma \times X$.  Therefore it has a subsequence with a
Gromov limit $u: C \to X$ where the stabilization $C^s$ of $C$ is a
fiber of $\ol{\cU}_\Gamma$ and $u$ is pseudoholomorphic for the
pull-back of the restriction of $J_\Gamma$ to $C^s$.  If we restrict
to sequences of maps $u_\nu: C_\nu \to X$ sending the markings to $V'$
or $V''$ then in fact $C^s$ is equal to $C$, since non-constant
components of $u$ with fewer than three markings are impossible.

We distinguish components of the curve that are collapsed under
forgetting the first or second group of markings.  Let
\[ f': \ol{\cU}_{\Gamma} \to \ol{\cU}_{\Gamma''}, \quad 
f'':\ol{\cU}_{\Gamma} \to \ol{\cU}_{\Gamma'} \]  
denote the forgetful maps forgetting the first $n'$ resp. last $n''$
markings and stabilizing.  Call a component of $C$ {\it $f'$-unstable}
if it is collapsed by $f'$, and {\it $f'$-stable} otherwise, in which
case it corresponds to a component of $f'(C)$.  $f''$-unstable
components are defined similarly.

\begin{definition} \label{plocal} 
{\rm (Local and partly local almost complex
    structures) } 
\begin{enumerate} 
 \item  A domain-dependent almost complex structure
\[ J_\Gamma: \ol{\cU}_\Gamma \to \J(X,V',V'') \]
is {\it local} if and only if for each $v \in \Ver(\Gamma)$ the restriction
$J_\Gamma | \ol{\cU}_{\Gamma,v}$ is local in the sense that
$J_\Gamma | \ol{\cU}_{\Gamma,v}$ is pulled back from some map
$J_{\Gamma,v}$ defined on the universal curve $\ol{\cU}_{\Gamma(v)}$
and equal to $J_{V',V''}$ near any special point of
$\ol{\cU}_{\Gamma,v}$.
\item A domain-dependent almost complex structure
 \[ J_\Gamma: \ol{\cU}_\Gamma \to \J(X,V',V'') \]
is {\it$f'$-local} if and only if 
\begin{enumerate} 
\item for each $v \in \Ver(\Gamma)$ such that $\cU_{\Gamma,v}$ is
  $f'$-stable (that is, has sufficiently many $V''$ markings) then
  $J_\Gamma | \ol{\cU}_{\Gamma,v}$ is local in the sense that
  $J_\Gamma | \ol{\cU}_{\Gamma,v}$ is pulled back from some map
  $J_{\Gamma,v}$ defined on the universal curve $\ol{\cU}_{\Gamma(v)}$
  and equal to $J_{V,V'}$ near any point $z \in C$ mapping to a
  special point $f'(z)$ of $f'(C)$, and
\item for each $v \in \Ver(\Gamma)$ such that $\cU_{\Gamma,v}$ is
  $f'$-unstable (that is, does not have sufficiently many $V''$
  markings) then $J_\Gamma | \ol{\cU}_{\Gamma,v}$ is constant on each
  fiber of $\ol{\cU}_{\Gamma,v}$.
\end{enumerate} 
The definition of $f''$-local is similar. 
\end{enumerate} 
\end{definition}

\begin{remark} \label{weaker} Note that $f'$-pullbacks
  $ (f')^* J_{\Gamma''}$ are $f'$-local, and local almost complex
  structures are $f'$-local.  The condition that an almost complex
  structure be $f'$-local is weaker than the condition that it be
  pulled back under $f'$, because the restriction
  $J_\Gamma | \ol{\cU}_{\Gamma,v}$ is allowed to depend on special
  points $z \in \ol{\cU}_{\Gamma,v}$ that are forgotten under
  $f'$. \end{remark}

\begin{remark} \label{pullback} One can reformulate the $f'$-local
  condition as a pullback condition for a forgetful map that forgets
  almost the same markings as those forgotten by $f'$.  Let $C$ be a
  curve of type $\Gamma$.  Let $C^{\us} \subset C$ be the locus
  collapsed by $f'$.  For each connected component
  $C_i, i = 1,\ldots, k$ of $C^{\us}$ mapping to a marking of $f'(C)$
  choose $j(i)$ so that $z_{j(i)} \in C_i$.  Let
  $I^{\us} \subset \{ 1, \ldots, n \}$ denote the set of indices $j$
  of markings $z_j \in C^{\us}$ with $z_j \neq z_{j(i)}, \forall i$.
  Forgetting the markings with indices in $I^{\us}$ and collapsing
  defines a map $f: C \to f(C)$ such that any collapsed component of
  $C$ maps to a special point of $f(C)$.  Let $\Gamma^f$ denote the
  combinatorial type of $f(C)$.  Then $J_\Gamma$ is $f'$-local if and
  only if $J_\Gamma = f^* J_{\Gamma^f}$ is pulled back from a local
  domain-dependent almost complex structure
  $J_{\Gamma^f}: \ol{\cU}_{\Gamma^f} \to \J(X,V',V'')$.  Indeed, the
  collapsed components under $C \to f(C)$ are the same as those of
  $f': C \to f'(C)$ since adding a single marking $z_i$ on the
  components that collapse $C_v$ to markings $f(C_v) \subset f(C)$
  does not stabilize $C_v$.  So the pull-back condition
  $J_\Gamma = f^* J_{\Gamma^f}$ requires $J_\Gamma$ to be constant on
  the components $C_v$ such that $\dim(f(C_v)) = 0$.  On the other
  hand, any irreducible component of $f(C)$ is isomorphic, as a stable
  marked curve, to an irreducible component of $C$ not collapsed under
  $f'$.
\end{remark} 

\begin{remark} \label{both} There also exist domain-dependent almost
  complex structures that are both $f'$ and $f''$-local.  Indeed
  suppose that $C$ is a curve of type $\Gamma$, and
  $K \subset \{1,\ldots, n' + n'' \} $ is the set of markings on
  components collapsed by $f'$ or $f''$.  Forgetting the markings
  $z_k, k \in K$ defines a forgetful map
  $f^{\s\s}: C \to f^{\s\s}(C)$, where $f^{\s\s}(C)$ is of some
  (possibly empty) type $\Gamma^{\s\s}$.  Let
  $J_{\Gamma^{\s\s}}: \ol{\cU}_{\Gamma^{\s\s}} \to \J(X,V',V'')$ be a
  domain-dependent almost complex structure for type $\Gamma^{\s\s}$.
  Then $(f^{\s\s})^* J_{\Gamma^{\s\s}}$ is both $f'$ and $f''$-local
  (taking the constant structure $J_{V',V''}$ if $\Gamma^{\s\s}$ is
  empty.)
\end{remark}

\begin{lemma} \label{extends} The space of $f'$-local
  resp. $f''$-local resp. $f'$ and $f''$-local almost complex
  structures tamed by or compatible with the symplectic form $\omega$
  is contractible.  Any $f'$-local resp. $f''$-local resp. $f'$ and
  $f''$-local $J_\Gamma | \partial \ol{\U}_\Gamma$ defined on the
  boundary
  $\partial \ol{\U}_\Gamma := \ol{\U}_\Gamma | \partial
  \ol{\M}_\Gamma$
  extends to a $f'$-local resp. $f''$-local resp. $f'$ and $f''$-local
  structure $J_\Gamma$ over an open neighborhood of the boundary
  $\partial \ol{\U}_\Gamma$ in $\ol{\U}_\Gamma$.
\end{lemma} 

\begin{proof} Contractibility follows from the contractibility of
  tamed or compatible almost complex structures.  Since the space of
  $f'$-local tamed almost complex structures is contractible, it
  suffices to show the existence of an extension of $J_\Gamma$ near
  any stratum $\ol{\U}_{\Gamma_1} \subset \ol{\U}_\Gamma$ and then
  patch together the extensions.  \cw{Local domain-dependent almost
  complex structures $J_\Gamma$ extend by a gluing construction in
  which open balls $U_+, U_-$ around a node are replaced by a
  punctured ball $V \cong U_+^\times \cong U_-^\times$ on which the
  almost complex structure is equal to the base almost complex
  structure $J_{V', V''}$.}

  In the partly-local case recall from Remark \ref{pullback} that
  $J_\Gamma$ is the pull-back of a local almost complex structure
  $J_{\Gamma^f}$ near any particular fiber of the universal curve.
  Define an extension of $J_\Gamma$ near curves of type $\Gamma_1$ by
  first extending $J_{\Gamma^f}$ and then pulling back.  In more
    detail, let $C$ be such a curve and let $C_1,\ldots, C_k$ denote
    the connected components of $C$ collapsed by $f'$ to a non-special
    point of $f'(C)$.  Choose a marking $z_i \in C_i$ and let
    $\Gamma^{\s}$ resp. $\Gamma_1^{\s}$ denote the type obtained from
    $\Gamma$ resp. $\Gamma_1$ by forgetting all markings on $C_i$
    except $z_i$, for each $i = 1,\ldots, k$.  Consider the forgetful
    map
  \[ f: \ol{\cU}_{\Gamma} \to \ol{\cU}_{\Gamma^f} \]
  that forgets all but the marking $z_i$ on $C_i$.  As discussed
in Remark \ref{pullback}
  $J_{\Gamma_1}$ is  the pullback of a complex structure
  \[ J_{\Gamma_1^f}: \ol{\cU}_{\Gamma_1^f} \to \J(X,V',V'') . \]
  Since the complex structure $J_{\Gamma_1^f}$ is constant equal to
  the base almost complex structure $J_{V',V''}$ near the nodes (which
  must join non-collapsed components) $J_{\Gamma_1^f}$ naturally
  extends to a domain-dependent almost complex structure
  $J_{\Gamma^f}$ on a neighborhood $\cN_{\Gamma_1^f}$ of
  $\ol{\cU}_{\Gamma_1^f}$ in $\ol{\cU}_{\Gamma^f}$ by taking
  $J_{\Gamma^f}$ to equal $J_{V',V''}$ near the nodes.  Now take
  $J_\Gamma = f^* J_{\Gamma^f}$ to obtain an extension of $J_\Gamma$
  from $\cU_{\Gamma_1}$ to a neighborhood $f^{-1}(\cN_{\Gamma_1^f})$.
  The proof for $f'$ local or $f'$ and $f''$-local structures is
  similar.
\end{proof}

\section{Transversality}
 
We wish to inductively construct partly-local almost complex
structures so that the moduli spaces of stable maps define
pseudocycles. Recall that the combinatorial type of a stable map is
obtained from the type of stable curve by decorating the vertices with
homology classes; we also wish to record the intersection
multiplicities with the Donaldson hypersurfaces.  More precisely, a
{\it type} of stable map $u$ from $C$ to $(X,V',V'')$ consists of a
type $\Gamma$ the stable curve $C$ (the graph with vertices
corresponding to components and edges corresponding to markings and
nodes) with the labelling of vertices $v \in \Ver(\Gamma)$ by homology
class $d(v) = [u |C_v] \in H_2(X)$, labelling of the semi-infinite
edges $e$ by either $V'$ or $V''$,\footnote{To obtain evaluation maps
  one should allow additional edges, but here we ignore evaluation
  maps.} and by the intersection multiplicities $m'(e), m''(e)$ with
$V'$ and $V''$ (possibly zero if the corresponding marking does not
map to $V'$ or $V''$.  A stable map is {\it adapted} of type $\Gamma$
if each connected component of $u^{-1}(V')$ resp.  $u^{-1}(V'')$
contains at least one marking $z_e$ corresponding to an edge $e$
labelled $V'$ resp. $V''$, and each marking $z_e$ maps to $V'$ or
$V''$ depending on its label.  A stable map is {\it adapted} of type
$\Gamma$ if
\begin{enumerate}
\item each connected component of $u^{-1}(V')$ resp. $u^{-1}(V'')$ contains at least one marking $z_e$ corresponding to an edge $e$ with labelling $m'(e) \geq 1$ resp. $m''(e) \geq 1$, and

\item if $m'(e) \geq 1$ resp. $m''(e) \geq 1$, then the marking $z_e$ is mapped to $V'$ resp. $V''$.
\end{enumerate}

By forgetting the extra data and stabilization one can associate to
each type of stable maps to a type of stable curves. In notation we do
not distinguish the two notions of types.  Given a type of stable map
$\Gamma$ choose a domain-dependent almost complex structure
$J_\Gamma$.  Denote by $\M_\Gamma(X,J_\Gamma)$ the moduli space of
adapted $J_\Gamma$-holomorphic stable maps $u: C \to X$ of type
$\Gamma$, such that for each $v \in {\rm Vert}(\Gamma)$ with
$d(v) \neq 0$, the image of $u_v$ is not contained in $V' \cup V''$,
and for each semi-infinite edge $e$ attached to $v$, the local
intersection number of $u_v$ with $V'$ resp. $V''$ at $z_e$ is equal
to $m'(e)$ resp. $m''(e)$.   \cw{The moduli space $\M_\Gamma(X,J_\Gamma)$
is locally cut out by a smooth map of Banach manifolds: Given a local
trivialization of the universal curve given by an subset
$\cM^i_\Gamma \subset \cM_\Gamma$ and a trivialization
$C \times \cM^i_\Gamma \to \cU_\Gamma^i = \cU_\Gamma
|_{\cM^i_\Gamma}$,
we consider the space of maps $\Map(C,X)_{k,p}$ of Sobolev class $k,p$
for $p \ge 2$ satisfying the above constraints and $k$ sufficiently
large to the space of $0,1$-forms with values in $TX$ given by the
Cauchy-Riemann operator $\olp_{J_\Gamma}$ associated to $J_\Gamma$.
The linearization of this operator is denoted $D_{u}$ (or
$D_{u,J_\Gamma}$ to emphasize dependence on $J_\Gamma$) and the map
$u$ is called {\it regular} if $D_{u}$ is surjective.}  We call a type
$\Gamma$ of stable map $u: C \to X$ {\it crowded} if there is a
maximal ghost subtree of the domain $C_1 \subset C $ with more than
one marking $z_e \in C_1$ and {\it uncrowded} otherwise.  It is not in
general possible to achieve transversality for crowded types using the
Cieliebak-Mohnke perturbation scheme.

\begin{definition} We say a domain-dependent almost complex structure
  $J_\Gamma$ is {\it regular} for a type of map $\Gamma$ if
  \begin{enumerate}
  \item if $\Gamma$ is uncrowded then every element of the moduli
    space $\M_\Gamma(X,J_\Gamma)$ of adapted $J_\Gamma$-holomorphic
    maps is regular; and
  \item If $\Gamma$ is crowded then there exists a regular
    $J_{\Gamma^{\s}}$ for some uncrowded type $\Gamma^{\s}$ obtained
    by forgetting all but one marking $z_e$ on each maximal ghost
    component for curves of type $\Gamma$ such that $J_{\Gamma^{\s}}$
    is equal to $J_\Gamma$ on all non-constant components, that is,
    all components of $\ol{\cU}_\Gamma$ on which the maps
    $u: C \to X$ in $\M_\Gamma(X,J_\Gamma)$ are non-constant.
\end{enumerate} 
\end{definition} 

Recall the construction by Floer \cw{ \cite[Lemma 5.1]{floer}} of a
subspace of smooth functions with a \cw{separable} Banach space
structure.  Let $\ul{\eps} = ( \eps_\ell, \ell \in \Z_{\ge 0})$ be a
sequence of constants converging to zero.  Let
$\J_\Gamma(X)_{\ul{\eps}}$ denote the space of domain-dependent almost
complex structures of finite Floer norm as in \cite[Section 5]{floer}.
In particular, $\J_\Gamma(X)_{\ul{\eps}}$ allows variations with
arbitrarily small support near any point.

\begin{proposition} 
\begin{enumerate} 
\item \label{c} For a regular domain-dependent almost complex
  structure $J_{\Gamma''}$ the pull-back $ (f')^* J_{\Gamma''}$ is
  regular, and similarly for the pull-back $(f'')^* J_{\Gamma'}$ for
  regular $J_{\Gamma'}$.
\item \label{d} Suppose that $J_\Gamma | \partial \ol{\cU}_\Gamma$ is
  $f'$-local and is a regular domain-dependent almost complex
  structure defined on the boundary
  $\partial \ol{\cU}_\Gamma \to \partial \ol{\M}_\Gamma$.  The set of
  regular $f'$-local extensions is comeager, that is, is the
  intersection of countably many sets with dense interiors.
\item \label{e} Any parametrized-regular homotopy
  $J_{\Gamma,t} | \partial \ol{\cU}_\Gamma$ between two regular
  $f'$-local domain-dependent almost complex structures
  $J_{\Gamma,0}, J_{\Gamma,1}$ on the boundary
  $\partial \ol{\cU}_\Gamma$ may be extended to a parametrized-regular
  one-parameter family of $f'$-local structures $J_{\Gamma,t}$ equal
  to $J_{\Gamma,t}$ over $\ol{\cU}_\Gamma$.
\end{enumerate} 
\end{proposition} 

\begin{proof}
  Item \eqref{c} is immediate from the definition, since any variation
  of $J_{\Gamma''}$ induces a variation of $(f')^* J_{\Gamma''}$.
  \eqref{d} is an application of Sard-Smale applied to a universal
  moduli space.  We sketch the proof which is analogous to that in
  Cieliebak-Mohnke \cite[Chapter 5]{cm:trans}.  By Lemma
  \ref{extends}, $J_\Gamma | \partial \ol{\U}_\Gamma$ has an extension
  over the interior.  For transversality, first consider the case of
  an uncrowded type $\Gamma$ of stable map.  \cw{ Choose open subsets
    $L_\Gamma,N_\Gamma \subset \ol{\cU}_\Gamma$ of the boundary resp.
    markings and nodes, such that $L_\Gamma$ is union of fibers of
    $\ol{\cU}_\Gamma$ containing the restriction
    $\ol{\cU}_\Gamma | \partial \M_\Gamma$ and $N_\Gamma$ is
    sufficiently small so that the intersection of the complement of
    $N_\Gamma$ with each component of each fiber of $\cU_\Gamma$ not
    meeting $L_\Gamma$ is non-empty.  Let $\M^{\univ}_\Gamma(X)$
    denote the universal moduli space consisting of pairs
    $(u,J_\Gamma)$, where $u: C \to X$ is a $J_\Gamma$-holomorphic map
    of some Sobolev class $W^{k,p}, kp \ge 3, p \ge 2$ on each
    component (with $k$ sufficiently large so that the given vanishing
    order at the Donaldson hypersurfaces $V',V''$ is well-defined).
    Let $\J^E_\Gamma(X,N_\Gamma,S_\Gamma) \subset \J^E_\Gamma(X)$
    denote the space of $J_\Gamma \in \J_\Gamma^E(X)_{\ul{\eps}}$ that
    are $f'$-local domain-dependent almost complex structures that
    agree with $J_{V',V''}$ on the neighborhood $N_\Gamma$ of the
    nodes and markings $z \in \ol{\cU}_\Gamma$ that map to special
    points $f'(z) \in \ol{\cU}_{f'(\Gamma)}$ as in Definition
    \ref{plocal}, and equal to the given extension in the neighborhood
    $L_\Gamma$ of the boundary, and constant on the components
    required by $f'$-locality in Definition \ref{plocal}.  By elliptic
    regularity, $\M^{\univ}_\Gamma(X)$ is independent of the choice of
    Sobolev exponents.

    The universal moduli space is a smooth Banach manifold by an
    application of the implicit function theorem for Banach manifolds.
    Let $\cU_\Gamma^i \to \cM_\Gamma^i , i = 1,\ldots, m$ be a
    collection of open subsets of the universal curve
    $\cU_\Gamma \to \cM_\Gamma$ on which the universal curve is
    trivialized via diffeomorphisms
    $\cU_\Gamma^i \to \cM_\Gamma^i \times C$.  The space of pairs
    $(u : C \to X, J_\Gamma)$ with $[C] \in \cM_\Gamma^i$, $u$ of type
    $\Gamma$ of class $W^{k,p}$ on each component, and
    $J_\Gamma \in \J^E_\Gamma(X,N_\Gamma,S_\Gamma)$ is a smooth
    separable Banach manifold.  Since we assume that $J_\Gamma$ is
    regular on the boundary $\partial \cU_\Gamma$, an argument using
    Gromov compactness shows that by choosing $L_\Gamma$ sufficiently
    small we may assume that $\ti{D}_{u,J_\Gamma}$ is surjective for
    $[C] \in L_\Gamma$, since regularity is an open condition in the
    Gromov topology \cite[Section 10.7]{ms:jh}.  Let
    $\ti{D}_{u,J_\Gamma}$ the linearization of
    $(u,J_\Gamma) \mapsto \olp_{J_\Gamma} u$, and suppose that $\eta$
    lies in the cokernel of $\ti{D}_{u,J_\Gamma}$.  We have
    $D_u^* \eta^s = 0$ where $D_u$ is the usual linearized
    Cauchy-Riemann operator \cite[p. 258]{ms:jh} for the map; in the
    case of vanishing constraints at the Donaldson hypersurfaces see
    Cieliebak-Mohnke \cite[Lemma 6.6]{cm:trans}.  By variation of the
    almost complex structure $J_\Gamma$ and unique continuation,
    $\eta$ vanishes on any component on which $u$ is non-constant.  On
    the other hand, for any constant component $u_v: C_v \to X$, the
    linearized Cauchy-Riemann operator $D_{u_v}$ on a trivial bundle
    $u_v^* TX$ is regular with kernel $\ker(D_{u_v})$ the space of
    constant maps $\xi: C_u \to (u_v)^* TX$.  It follows by a standard
    inductive argument that the same holds true for a tree
    $C' = \cup_{v \in V} C_v , \d u |_{C'} = 0 $ of constant
    pseudoholomorphic spheres so the element $\eta$ vanishes on any
    component $C_v \subset C$ on which $u$ is constant.  It follows
    that $\cM^{\univ,i}_\Gamma(X)$ is a smooth Banach manifold.  For a
    comeager subset $\J_\Gamma^{\reg}(X) \subset \J_\Gamma(X)$ of
    partly almost complex structures in the space above, the moduli
    spaces $\cM^i_\Gamma(X) = \cM_\Gamma(X) |_{\cM^i_\Gamma}$ are
    transversally cut out for each $i = 1,\ldots, m$.  The transition
    maps between the local trivializations
    $\cM_\Gamma^i \cap \cM_\Gamma^j \to \Aut(C)$ induce smooth maps
    $\cM_\Gamma^{i}(X) |_{\cM_{\Gamma}^i \cap \cM_\Gamma^j} \to
    \cM_\Gamma^{j}(X) |_{\cM_{\Gamma}^i \cap \cM_\Gamma^j}$
    making $\cM_\Gamma(X)$ into a smooth manifold.}

  Next consider a crowded type $\Gamma$.  Let
  $f: \Gamma \to f(\Gamma)$ be a map forgetting all but one marking on
  each maximal ghost component $C' \subset C$ and stabilizing; the
  multiplicities $m'(e),m''(e)$ at any marking $z_e$ is the sum of the
  multiplicities of markings in its preimage $f^{-1}(z_e)$.  Define
  $J_{\Gamma^f}$ as follows.
\begin{enumerate} 
\item If $\cU_{\Gamma^f,v} \cong \cU_{\Gamma,v}$ let
  $J_{\Gamma^f} | \cU_{\Gamma^f,v}$ be equal to
  $J_{\Gamma} | \cU_{\Gamma,v}$.
\item Otherwise let
  $J_{\Gamma^f} : \cU_{\Gamma^f,v}\to \J^E(X,V',V'')$ be constant
  equal to $J_{V',V''}$.
\end{enumerate}
\noindent The map $J_{\Gamma^f}$ is continuous because any non-collapsed ghost
component $C_v \subset C$ must connect at least two non-ghost
components $C_{v_1} , C_{v_2} \subset C$ and the connecting points of
the non-ghost components $f'(C_{v_1}), f'(C_{v_2})$ is a node of the
curve $f'(f(C))$ of type $f'(\Gamma^f)$.  For a comeager subset of
$J_\Gamma$ described above, the complex structures $J_{\Gamma^f}$ are
also regular by the argument for uncrowded types.  Item \eqref{e} is a
parametrized version of \eqref{d}.
\end{proof} 

\begin{corollary} There exists a regular homotopy
  $J_{\Gamma,t}, t \in [-1,1]$ between $ (f'')^* J'_{\Gamma'}$ and
  $ \ (f')^* J''_{\Gamma''}$ in the space of maps
  $\ol{\U}_\Gamma \to \J(X,V',V'')$ that are $f'$-local for
  $t \in [-1,0]$ and $f''$-local for $t \in [0,1]$
\end{corollary}

\begin{proof}  
  Let $\ul{J} = (J_\Gamma)$ be a collection of regular
  domain-dependent almost complex structures that are both $f'$ and
  $f''$-local, 
as in Remark \ref{both}.   By part \eqref{d} above, for each type $\Gamma$ there
  exists a regular homotopy from $J_\Gamma$ to $(f')^* J_{\Gamma''}$
  resp. $(f'')^* J_{\Gamma'}$ extending given homotopies on the
  boundary.  The existence of a regular homotopy now follows by
  induction.
\end{proof}

\def\cprime{$'$} \def\cprime{$'$} \def\cprime{$'$} \def\cprime{$'$}
\def\cprime{$'$} \def\cprime{$'$}
\def\polhk#1{\setbox0=\hbox{#1}{\ooalign{\hidewidth 
      \lower1.5ex\hbox{`}\hidewidth\crcr\unhbox0}}} \def\cprime{$'$}
\def\cprime{$'$} \def\cprime{$'$} \def\cprime{$'$}

 \end{document}